\documentclass[a4paper,11pt]{elsarticle}
\usepackage{amsmath,amssymb,amsfonts}
\usepackage{amsthm}
\usepackage{graphicx,color}
\usepackage{mathtools}
\usepackage{enumerate}
\newtheorem{lemma}{Lemma}

\newtheorem{theorem}{Theorem}

\numberwithin{equation}{section}
\journal{Applied Mathematics Letters}
\begin{document}
\nocite{*}
\begin{frontmatter}
\title{Solution Space Characterisation of Perturbed Linear Volterra Integrodifferential Convolution Equations: the $L^p$ case}

\author{John A. D. Appleby$^a$}
\ead{john.appeby@dcu.ie}
\cortext[cor1]{Corresponding author}

\author{Emmet Lawless$^a$\corref{cor1}}
\ead{emmet.lawless6@mail.dcu.ie}

\affiliation{School of Mathematical Sciences, Dublin City University.}
\address{Dublin City University, DCU Glasnevin Campus, School of Mathematical Sciences, Dublin 9, Ireland.}

\begin{abstract}
In this paper we characterise the $L^p$ stability of perturbed linear Volterra integrodifferential convolution equations. Additionally we provide a framework which points to necessary and sufficient conditions on the forcing function that ensures the solution lies in a particular function space. We highlight how such a result is of interest when studying perturbed Stochastic Functional Differential Equations (SFDEs).
\end{abstract}
\end{frontmatter}

\section{Introduction}
Over the past few decades, Volterra integral equations have attracted a great deal of study, traditionally being applied in fields such as population dynamics or epidemic modelling (see e.g. the monograph of Cushing~\cite{Cush}). More recently there has been a surge of research on stochastic Volterra equations in finance in the form of volatility models (see e.g. 
Jaber et. al. \cite{Jab2019}). Although the field has progressed tremendously, there still remains some unanswered questions about classical perturbed deterministic equations, which we aim to address in this paper. The equation of interest is the following\footnote{For ease of notation and to keep calculations compact we consider only the scalar case, but the reader will observe all proofs can easily be extended to finite dimensions with no extra complications.}
\begin{equation} \label{eq. x dynamics}
    x'(t) = \int_{[0,t]}\nu(ds)x(t-s)+f(t), \quad t \geq 0, \\
\end{equation}
where $\nu$ is a finite Borel measure and $x(0)=\xi \in \mathbb{R}$ is a given initial condition. Convolution Volterra equations have been well--understood for some time, particularly when considering linear equations, so we do not provide an exhaustive list of the literature but instead refer the reader to the classical monographs by Gripenberg et al.~\cite{GLS} and Corduneanu \cite{Cor90b} for a thorough exposition of the theory. It is well--known that knowledge of the underlying differential resolvent given by $r'(t)=\int_{[0,t]}\nu(ds)r(t-s),$ for $t>0$ (with $r(0)=1$) is critical if we want to understand the behaviour of \eqref{eq. x dynamics}. The most comprehensive, and very representative, type of perturbation result for equation \eqref{eq. x dynamics} is the following: under the assumption\footnote{We make this necessary assumption to obtain sharp stability results throughout the work.} that $r \in L^1(\mathbb{R}_+)$, a \emph{sufficient} condition for solutions of \eqref{eq. x dynamics} to be in some (reasonable) function space $V$ is for $f \in V$. Such reasonable spaces are numerous and listed in \cite[Thm. 3.3.9]{GLS}. However, this leaves open the question of what conditions on $f$ are \emph{necesssary and sufficent} for the solution to be in $V$. This topic is deeply related to the question of so-called \emph{admissibilty} of (Volterra) operators and of pairs of function spaces. For perturbed Volterra equations, sufficient conditions for solutions to lie in a space $V$, when forcing functions lie in a space $W$ have been extensively discussed in Miller 
\cite[Ch. 5]{miller:1971} and Corduneanu \cite[Ch.2]{cor:1973non}. 

One facet of admissibility theory, which is particularly germane to this paper, is that it often dispels the notion that if  solutions of \eqref{eq. x dynamics} are to be in some space $V$, the ``worst behaving'' perturbing function that is permitted must also lie in $V$. In fact, the case where $V=BC_0([0,\infty),\mathbb{R})$ has already been indirectly proven in the literature (here $BC_0([0,\infty),\mathbb{R})$ is understood to be the space of bounded, continuous functions from $\mathbb{R}_+$ to $\mathbb{R}$ that decay to zero at infinity). The result is as follows: under the assumption that $x$ solves \eqref{eq. x dynamics} and $r \in L^1(\mathbb{R}_+)$ then the following are equivalent:
\begin{align*}
    & (\textbf{A})\quad  x(t;\xi) \in BC_0([0,\infty),\mathbb{R}) \text{ for all } \xi \in \mathbb{R},\\ 
    & (\textbf{B}) \quad F(t;\theta):\mathbb{R}_+ \to \mathbb{R} : t \mapsto \int_t^{t+\theta}f(s)ds \in BC_0([0,\infty),\mathbb{R}) \text{ for each } \theta \in (0,1].
\end{align*} 
That (\textbf{B}) $\implies$ (\textbf{A}) is a consequence of Theorem 11.4.3 in~\cite{GLS} and (\textbf{A}) $\implies$ (\textbf{B}) is an obvious consequence of equation \eqref{eq. integrated x} below. The key result which enables such a sharp condition is Lemma 15.9.2 in \cite{GLS}.  The authors believe this Lemma can be generalised, leading to results roughly of the form $x\in V$ iff $F(\cdot;\theta)\in V$ for all $\theta\in (0,1]$ for many spaces $V$, including spaces of unbounded functions (see e.g. \cite{AP:2017} for summation equations). Conditions of type (\textbf{B}) on forcing terms date back at least to papers of Strauss and Yorke~\cite{SY67a,SY67b}, concerning the stability theory of asymptotically autonomous deterministic differential equations.   

The main aim of this paper is to give a framework to characterise when solutions of \eqref{eq. x dynamics} lie in a space $V$. We outline this procedure when $V=L^p(\mathbb{R}_+)$, settling a conjecture for stochastic equations in \cite{AL}. We show that solutions of \eqref{eq. x dynamics} are in $L^p(\mathbb{R}_+)$  if and only if 
\begin{equation} \label{stability condition on f}
    \left\Vert \int_{\cdot}^{\cdot+\theta}f(s)ds \right\Vert_{L^p(\mathbb{R}_+)} < +\infty, \quad \text{for each } \theta \in (0,1].
\end{equation}
This ``interval average" condition greatly expands the class of perturbing functions which preserve stability. An explicit example is given of a function exhibiting unbounded oscillatory behaviour which satisfies \eqref{stability condition on f} but is not in $L^p(\mathbb{R}_+)$ for any $p$. Furthermore, conditions on the interval average are crucial in the analysis of perturbed SFDEs. In  \cite{AL} the authors have provided a collection of mean square stability results for the equation 
\begin{equation} \label{eq. Stochastic X}
    dX(t)=\left(f(t)+\int_{[0,\tau]}X(t-s)\nu(ds)\right)dt+\left(g(t)+\int_{[0,\tau]}X(t-s)\mu(ds)\right)dB(t),
\end{equation}
where $\nu$ and $\mu$ are finite Borel measures, $\tau > 0$ is a fixed constant and $B$ is a one dimensional Brownian motion. The interval average condition on the deterministic perturbing functions $f$ and $g$ proves necessary for sharp mean square stability results. Indeed, granted the asymptotic stability of the underlying stochastic equation without perturbations (see e.g. \cite{AMR} for unperturbed stability conditions), it is shown that $f$ and $g^2$ must obey condition (\textbf{B}) above, if the mean square is to obey $\mathbb{E}[X^2(t)]\to 0$ as $t\to\infty$. If we wish $\mathbb{E}[X^2]$ to be in $L^1(\mathbb{R}_+)$, the necessary and sufficient conditions on $f$  and $g$ are $g\in L^2(\mathbb{R}_+)$ and $\int_0^t e^{-(t-s)}f(s)\,ds \in L^2(\mathbb{R}_+)$. Theorem \ref{thm. x L^p stability} below with $\nu(dt)=-\delta_{0}(dt)$ shows that the second condition is equivalent to \eqref{stability condition on f} with $p=2$ (here $\delta_0$ is the unit point mass at zero). It is open what happens in the (stochastic) Volterra case: this will be addressed in a future work.

\section{Results}
If $M(\mathbb{R}_+)$ is the space of finite signed Borel measures on $\mathbb{R}_+$, and $\nu\in M(\mathbb{R}_+)$, we consider the halfline Volterra equation given by
\begin{equation} \label{eq. x}
    x'(t)  = \int_{[0,t]}\nu(ds)x(t-s)+f(t), \quad t \geq 0; \quad    x(0)  = \xi,
\end{equation}
where $\xi \in \mathbb{R}$ is the initial condition. We say $x$ is a solution of \eqref{eq. x} on an interval $(0,T]$ whenever $x$ is locally absolutely continuous, satisfies the initial condition and obeys the dynamics  in \eqref{eq. x} for almost all $t \in (0,T]$. We stipulate throughout, and \emph{without further reference}, that $f \in L^1_{loc}(\mathbb{R}_+)$, which guarantees that such a solution $x$ exists (see Gripenberg et al.\cite[Thm. 3.3.3]{GLS}). In particular, the solution satisfies a variation of constants formula
\begin{equation} \label{eq. VOC x}
    x(t)=r(t)\xi+\int_0^tr(t-s)f(s)ds, \quad t \geq 0,
\end{equation}
where $r$ is the so--called differential resolvent of $\nu$, which is the unique absolutely continuous function on $\mathbb{R}_+$ satisfying
\begin{equation} \label{eq. r}
  r'(t)=\int_{[0,t]}\nu(ds)r(t-s), \quad t>0; \quad r(0)=1. 
\end{equation}

Before we present the main result of this paper we prepare two lemmas, the first of which is required in the proof of the second. 
\begin{lemma} \label{lem. L^p equiv UB}
Suppose $\int_{\cdot}^{\cdot+\theta}f(s)ds \in L^p(\mathbb{R}_+) \text{ for all } \theta \in (0,1]$. Then,
\begin{equation} \label{eq.unifLp}
    \left\Vert \int_{\cdot}^{\cdot+\theta}f(s)ds \right\Vert_{L^p(\mathbb{R}_+)} < B, \quad \text{ for all } \theta \in (0,1].
\end{equation}
\end{lemma}

Note that the upper bound $B$ is independent of the parameter $\theta$. This is required if one wishes to bound $|F(\cdot;\theta)|^p$ while integrating over $\theta$. Before proving Lemma \ref{lem. L^p equiv UB} we remark that its reverse implication is obviously true and will be used in proofs of later results.

\begin{proof}[Proof of Lemma \ref{lem. L^p equiv UB}]
By hypothesis, for all $\theta \in (0,1]$ we have that,
\begin{equation*}
    \varphi(\theta) \coloneqq \int_{0}^\infty |F(t;\theta)|^pdt < +\infty.
\end{equation*}
Moreover, since $f$ is locally integrable, it follows that $|F|^p$ is indeed a well--defined measureable function and moreover the non-negativity of $|F|^p$ ensures that $\varphi$ is also measurable; this holds by e.g. Theorem 8.8 in Rudin \cite{Rudin2}. Thus by hypothesis for each $m \in \mathbb{N}$, the set
$Q_m=\left\{\theta\in[0,1]: \varphi(\theta) \leq m \right\}$ is well--defined and measurable. Notice also that $Q_m \subseteq Q_{m+1}$ and that the hypothesis ensures that $\bigcup_{m=1}^\infty Q_m =[0,1].$ This ensures that there exists at least one $m' \in \mathbb{N}$ such that the set $Q_{m'}$ has positive Lebesgue measure. In light of this we fix $m=m'$. Then applying Lemma 15.9.3 in Gripenberg et al. \cite{GLS} to the set $Q_m$ ensures the set $Q_m-Q_m \coloneqq \left\{\theta-\theta':\theta,\theta' \in Q_m\right\}$, contains an interval $(-\epsilon,\epsilon).$ As both $\theta$,$\theta' \in Q_m$ we have that $\varphi(\theta) \leq m$ and $\varphi(\theta') \leq m$. Thus without loss of generality we take $\theta > \theta'$ and by extending $f$ to be equal to zero on the interval $[-1,0)$ we can say for $t\geq0$,
\begin{align*}
    \left| \int_{t}^{t+\theta-\theta'} f(s)ds\right| & \leq \left| \int_{t-\theta'}^{t+\theta-\theta'} f(s)ds\right| + \left| \int_{t-\theta'}^{t} f(s)ds\right|.
\end{align*}
Using the inequality $(a+b)^p \leq 2^{p-1}(a^p+b^p)$ for $a,b \geq 0$, integrating both sides and letting $g(t) := |\int_0^tf(s)ds|^p$ gives 
\begin{align*}
    \frac{1}{2^{p-1}}\int_0^\infty\left| \int_{t}^{t+\theta-\theta'} f(s)ds\right|^p dt & \leq \int_0^\infty\left| \int_{t-\theta'}^{t+\theta-\theta'} f(s)ds\right|^p dt + \int_0^\infty\left| \int_{t-\theta'}^{t} f(s)ds\right|^p dt \\
    & = \int_{-\theta'}^0g(\tau+\theta) d\tau +\varphi(\theta) + \int_{-\theta'}^0 g(\tau+\theta') d\tau +\varphi(\theta')\\
    & \leq \sup_{-\theta' \leq \tau \leq 0}g(\tau+\theta) + \sup_{-\theta' \leq \tau \leq 0}g(\tau+\theta') +2m \\
    & \leq 2\left(\sup_{0\leq t \leq 1}g(t) +m\right).
\end{align*}
Hence, for all $T \in (-\epsilon,\epsilon)$ we have
\begin{equation*}
    \int_0^\infty\left| \int_{t}^{t+T} f(s)ds\right|^p dt \leq 2^p\left(\sup_{0\leq t \leq 1}\left|\int_{0}^{t} f(s)ds\right|^p +m\right) \eqqcolon B_1.
\end{equation*}
Note that $B_1$ is independent of $T$, and take any $\theta\in (0,1]$. There is a minimal $N=N(\epsilon) \in \mathbb{N}$ such that $N\epsilon \geq 2$, and $\theta \in [n\frac{\epsilon}{2},(n+1)\frac{\epsilon}{2}]$ for exactly one $n \in \left\{0,\ldots,(N-1)\right\}$. Write next 
\begin{equation*}
    \int_t^{t+\theta}f(s)ds = \sum_{j=1}^N \int_{t+\frac{(j-1)\theta}{N}}^{t+\frac{j\theta}{N}}f(s)ds=\sum_{j=1}^{N}\int_{t+(j-1)\theta^{\ast}}^{t+j\theta^{\ast}}f(s)ds,
\end{equation*}
where $\theta^{\ast}:=\frac{\theta}{N} \leq \frac{n+1}{2N}\epsilon \leq \frac{\epsilon}{2} < \epsilon$. Thus
\begin{equation*}
    \left| \int_t^{t+\theta}f(s)ds \right|^p \leq N^{p-1}\sum_{j=1}^N \left| \int_{t+(j-1)\theta^{\ast}}^{t+(j-1)\theta^{\ast}+\theta^{\ast}}f(s)ds\right|^p,
\end{equation*}
and since $\theta^{\ast} \in (0,\epsilon)$, for any $\theta \in (0,1]$ we have
\begin{equation*}
 \int_0^\infty \left|\int_t^{t+\theta}f(s)ds \right|^p dt \leq  N\cdot B_1 \eqqcolon B_p.
\end{equation*}
Since $B_p$ depends only on $N$ and $B_1$, and both are $\theta$--independent, the proof is complete. 
\end{proof}
Both Lemmas \ref{lem. L^p equiv UB} and \ref{lem. f decomposition} are adapted from a decomposition result for a function $f$ in which $F(t;\theta) \to 0$ as $t \to \infty$ for all $\theta \in [0,1]$. This result can be found as Lemma 15.9.2 in~\cite{GLS}.
\begin{lemma} \label{lem. f decomposition}
    Suppose $\Vert \int_\cdot^{\cdot+\theta}f(s)ds \Vert_{L^p(\mathbb{R}_+)}<+\infty$ for all $\theta \in (0,1]$. Then $f=f_1+f_2$ where $f_1 \in L^p(\mathbb{R}_+)$ and $f_2$ is such that $\int_0^tf_2(s)ds \in L^p(\mathbb{R}_+)$
\end{lemma}

\begin{proof}[Proof of Lemma \ref{lem. f decomposition}]
Take $f_1(t)=\int_{t-1}^tf(s)ds$, $t\geq0$. Here we consider an extended version of $f$ where $f(t)\equiv 0$, for all $t<0$. Then $f_1$ is continuous and in $L^p(\mathbb{R}_+)$. Next we let $f_2=f-f_1$ for $t\geq0$. Then for $t\geq1$ we have
\begin{equation} \label{eq.key1}
    \int_0^tf_2(s)ds = \int_0^tf(s)ds-\int_0^t\int_{s-1}^sf(u)duds=\int_0^1\int_{t+v-1}^tf(u)dudv,
\end{equation}
the last identity following from a routine but tedious set of calculations, involving repeated application of Fubini's theorem. By Jensen's inequality, we get
\begin{equation*}
\left|\int_0^tf_2(s)ds \right|^p \leq \int_0^1\left|\int_{t+v-1}^tf(u)du\right|^p dv.
\end{equation*}
Integrating over $t$ on both sides and using Fubini's theorem, it follows that,
\begin{equation*}
    \int_1^\infty \left|\int_0^tf_2(s)ds \right|^p dt \leq \int_0^1\left(\int_0^\infty\left|\int_t^{t+v}f(u)du\right|^pdt\right)dv \leq \int_0^1Bdv =B,
\end{equation*}
where the constant $B$ is taken from Lemma \ref{lem. L^p equiv UB}. If $t\in[0,1]$ then the fact that $f_2$ is locally bounded ensures $\int_0^1 \left|\int_0^tf_2(s)ds \right|^p dt$ is finite.
\end{proof}
We are now in a position to state and prove our main result. By Lemma \ref{lem. L^p equiv UB}, the condition (\textbf{A}) below  is equivalent to \eqref{eq.unifLp}, but the latter condition is harder to check than (\textbf{A}).
\begin{theorem} \label{thm. x L^p stability}
    Let $x$ be the solution of \eqref{eq. x} and suppose $r \in L^1(\mathbb{R}_+)$. Then for $p \geq 1$, the following are equivalent:
    \begin{enumerate}
        \item[(\textbf{A})]  $\left\Vert \int_{\cdot}^{\cdot+\theta}f(s)ds \right\Vert_{L^p(\mathbb{R}_+)} < +\infty, \quad \text{for each } \theta \in (0,1]$;
        \item[(\textbf{B})] $x(\cdot,\xi) \in L^p(\mathbb{R}_+)$ for all $\xi \in \mathbb{R}$.
    \end{enumerate}    
\end{theorem}
\begin{proof}[Proof of Theorem \ref{thm. x L^p stability}]
First we show $(\textbf{A}) \implies (\textbf{B}).$ Recall that $x$ obeys a variation of constants formula given by $x(t,\xi)=r(t)\xi+(r\ast f)(t)$ for $t\geq0$, where $\ast$ denotes convolution of functions on $\mathbb{R}_+$. As $\nu \in M(\mathbb{R}_+)$ and $r \in L^1(\mathbb{R}_+)$ we have that $r' \in L^1(\mathbb{R}_+)$ and also $r(t)\to0,$ as $t\to \infty$. These facts, along with $r$ being absolutely continuous and 
well--behaved at both zero and infinity, ensure that $r\in L^p(\mathbb{R}_+)$ for all $p\geq1$. Thus $r(\cdot)\xi \in L^p(\mathbb{R}_+)$, so we need only focus on $x_1(t) \coloneqq (r\ast f)(t)$. Defining $f_3(t):=\int_0^tf_2(s)ds$, writing $f=f_1+f_2$ and integrating by parts we get
\begin{equation} \label{eq.key2}
    x_1(t) = (r\ast f_1)(t) +f_3(t)+(r' \ast f_3)(t), \quad t\geq 0.
\end{equation}
As $r, r' \in L^1(\mathbb{R}_+)$ and $f_1, f_3 \in L^p(\mathbb{R}_+)$ (by Lemma \ref{lem. f decomposition} and (\textbf{A})), an application of Theorem 2.2.2 in~\cite{GLS} ensures $x_1 \in L^p(\mathbb{R}_+)$ and the claim is proven.\\
To show (\textbf{B}) $\implies$ (\textbf{A}), take $\theta \in (0,1]$ and integrate \eqref{eq. x} over the interval $[t,t+\theta]$ to get
\begin{equation}  \label{eq. integrated x}
    x(t+\theta)-x(t)-\int_t^{t+\theta}(\nu \ast x)(s)ds=\int_t^{t+\theta}f(s)ds.
\end{equation}
 The first two terms on the left are in $L^p(\mathbb{R}_+)$. For the third, with $\frac{1}{p}+\frac{1}{q}=1$, by H\"{o}lder's inequality we get
\begin{align*}
    \int_0^\infty \left|\int_t^{t+\theta}(\nu \ast x)(s)ds \right|^pdt 
    & \leq \theta^{\frac{p}{q}}\int_0^\infty \int_t^{t+\theta}\left|(\nu \ast x)(s)\right|^pds dt\\
    &\leq \int_0^\infty \int_{\max{(s-\theta,0)}}^{s}\left|(\nu \ast x)(s)\right|^pdtds 
    \leq \int_0^\infty \left|(\nu \ast x)(s)\right|^pds.
\end{align*}
The last term is finite by virtue of Theorem 3.6.1 in Gripenberg et al.~\cite{GLS}. Thus, the left-hand side of \eqref{eq. integrated x} is in $L^p(\mathbb{R}_+)$, and therefore the righthand side is also, as was required. 
\end{proof}
The three key identities in extending this result to other spaces are \eqref{eq.key1}, \eqref{eq.key2} and for converse results, \eqref{eq. integrated x}. These identities hold, regardless of conditions on $f$, and show how the solution $x$ can be related to $f$ purely through the interval average $F(\cdot;\theta)$.  

\section{Example}
We consider the equation \eqref{eq. x} with  $f(t)=e^{\alpha t}\sin({e^{\beta t}})$ for $t\geq 0$, 
where $\nu$ and $\xi$ are defined as in the previous section and the parameters $\alpha$ and $\beta$ obey $0<\alpha<\beta$. Assume that the resolvent of the solution of \eqref{eq. x} is in $L^1(\mathbb{R}_+)$. In this case $f$ exhibits high frequency and rapidly growing oscillations. It is easily shown (see \cite{AL}) that 
$\int_t^{t+\theta}f(s)ds = O(e^{-(\beta-\alpha)t})$ as $t \to \infty$ for each $\theta>0$. This exponential decay ensures $\int_t^{t+\theta}f(s)ds \in L^p(\mathbb{R}_+)$ for any choices of $p\geq 1$ and $\theta>0$. Thus Theorem \ref{thm. x L^p stability} tells us that $x \in L^p(\mathbb{R}_+)$ despite the fact that $\int_0^\infty |f(s)|^pds = \infty$ for any choice of $p\geq 1$.
\newline

\textbf{Acknowledgements:} EL is supported by Science Foundation Ireland (16/IA/4443). JA is supported by the RSE Saltire Facilitation Network on Stochastic Differential Equations: Theory, Numerics and Applications (RSE1832).

\bibliographystyle{unsrt}

\end{document}